\documentclass{article}
\usepackage{arxiv}

\usepackage[utf8]{inputenc} % allow utf-8 input
\usepackage[T1]{fontenc}    % use 8-bit T1 fonts
\usepackage{hyperref}       % hyperlinks
\usepackage{booktabs}       % professional-quality tables
\usepackage{nicefrac}       % compact symbols for 1/2, etc.
\usepackage{microtype}      % microtypography

\usepackage{comment}
\usepackage{enumitem}
\usepackage{graphicx}
\usepackage{subfigure}
\usepackage{amssymb,amsmath,amsfonts}
\usepackage{amsthm}
\usepackage{algorithm}
\usepackage{algorithmic}
\usepackage{cleveref,wrapfig}
\usepackage{color, colortbl}
\usepackage[export]{adjustbox}
\newtheorem{assumption}{Assumption}[section]
\newtheorem{remark}{Remark}[section]

\newtheorem{lemma}{Lemma}[section]
\newtheorem{definition}{Definition}[section]
\newtheorem{theorem}{Theorem}[section]
\definecolor{LightCyan}{rgb}{0.88,1,1}

\renewcommand{\paragraph}[1]{\vspace{0.1em}\noindent \textbf{#1}}
\title{Asynchronous Optimization Methods for Efficient Training of Deep Neural Networks with Guarantees}
\date{}
\author{	
	Vyacheslav Kungurtsev\thanks{Support for this author was provided by the OP VVV project CZ.02.1.01/0.0/0.0/16\_019/0000765 
		``Research Center for Informatics''} \\
	Department of Computer Science\\
	Czech Technical 
	University in Prague\\
	\texttt{kunguvya@fel.cvut.cz}
	\And
	Malcolm Egan \\University of Lyon\\ INSA Lyon, INRIA \\{\tt malcom.egan@inria.fr}
	\And
	Bapi Chatterjee\thanks{Supported by the European Union's Horizon 2020 
		research and innovation programme under the Marie Skodowska-Curie grant 
		agreement No. 754411 (ISTPlus).}\\
	Institute of Science and Technology\\ 
	Austria \\
	\texttt{bapi.chatterjee@ist.ac.at}\\
	\And
	Dan Alistarh\thanks{This project has received funding from the European Research Council (ERC) under the European Union’s Horizon 2020 research and innovation programme (grant agreement No 805223).}\\
	Institute of Science and Technology\\ Austria \\
	\texttt{dan.alistarh@ist.ac.at} \\
}

\usepackage{xspace}
\usepackage{xcolor}

\newcommand{\asm}{\textbf{ASSM}\xspace}
\newcommand{\pasm}{\textbf{PASSM}\xspace}
\newcommand{\pasmp}{\textbf{PASSM$+$}\xspace}
\newcommand{\sgd}{\textbf{SGD}\xspace}

\begin{document}

%\RUNAUTHOR{Kungurtsev, Egan, Chatterjee, Alistarh}
%\RUNTITLE{Asynchronous Stochastic Subgradient Methods}

%\TITLE{Asynchronous Stochastic Subgradient Methods for General Nonsmooth Nonconvex Optimization}
%\ARTICLEAUTHORS{
%\AUTHOR{Vyacheslav Kungurtsev} \AFF{Department of Computer Science, 
%              Faculty of Electrical Engineering 
%              Czech Technical University in Prague 
%              ({\tt vyacheslav.kungurtsev@fel.cvut.cz}). \footnote{Support for this author was provided by the OP VVV project CZ.02.1.01/0.0/0.0/16\_019/0000765 
%``Research Center for Informatics''}}
%\AUTHOR{Malcolm Egan} \AFF{University of Lyon, INSA Lyon, INRIA {\tt malcom.egan@inria.fr} }

%\AUTHOR{Bapi Chatterjee} \AFF{IST Austria {\tt bapi.chatterjee@ist.ac.at}} \AUTHOR{Dan Alistarh} \AFF{IST Austria {\tt dan.alistarh@ist.ac.at}
%}}

\maketitle

%\ABSTRACT{
\begin{abstract}
Asynchronous distributed algorithms are a popular way to reduce synchronization costs in large-scale optimization, and in particular for neural network training. However, for nonsmooth and nonconvex objectives, few convergence guarantees exist beyond cases where closed-form proximal operator solutions are available. 
As most popular contemporary deep neural networks lead to nonsmooth and nonconvex objectives, there is now a pressing need for such convergence guarantees. In this paper, we analyze for the first time the convergence of stochastic asynchronous optimization for this general class of objectives. In particular, we focus on stochastic subgradient methods allowing for block variable partitioning, where the shared-memory-based model is asynchronously updated by concurrent processes. To this end, we first introduce a probabilistic model which captures key features of real asynchronous scheduling between concurrent processes; under this model, we establish convergence with probability one to an invariant set for stochastic subgradient methods with momentum. 

From the practical perspective, one issue with the family of methods we consider is that it is not efficiently supported by machine learning frameworks, as they mostly focus on distributed data-parallel strategies. 
To address this, we propose a new implementation strategy for  shared-memory based training of deep neural networks, whereby concurrent parameter servers are utilized to train a partitioned but shared model in single- and multi-GPU settings. 
Based on this implementation, we achieve on average ${\sim}1.2$x speed-up in comparison to state-of-the-art training methods for popular image classification tasks without compromising accuracy. 

\end{abstract}

%Asynchronous distributed methods are a popular way to reduce the communication and synchronization costs of large-scale optimization. Yet, for all their success, little is known about their convergence guarantees in the challenging case of general non-smooth, non-convex objectives, beyond cases where closed-form proximal operator solutions are available.
%This is all the more surprising since these objectives are the ones appearing in the training of deep neural networks.

%In this paper, we introduce the first convergence analysis covering asynchronous methods in the case of general non-smooth, non-convex objectives. Our analysis applies to stochastic sub-gradient descent methods both with and without block variable partitioning, and both with and without momentum. It is phrased in the context of a general probabilistic model of asynchronous scheduling accurately adapted to modern hardware properties. We validate our analysis experimentally in the context of training deep neural network architectures. We show their overall successful asymptotic convergence as well as exploring how momentum, synchronization, and partitioning all affect performance.
%}

\section{Introduction}
Training deep neural networks is a difficult problem in several respects \cite{goodfellow2016deep}. First, with multiple layers and nonlinear activation functions such as sigmoid and softmax functions, the resulting optimization problems are nonconvex. Second, ReLU activation functions and max-pooling in convolutional networks induce nonsmoothness, i.e., the objective is not differentiable everywhere. Finally, in applications it is often unreasonable to store entire data sets in memory in order to compute the objective or subgradients. As such, it is necessary to exploit stochastic methods. 

Machine learning applications, including training deep neural networks, have also motivated optimization algorithms that use high performance computing parallel architectures. In this paper, we focus on the shared-memory paradigm, although, as we show, our results can be efficiently extended to realistic distributed settings. Recent interest in this topic was sparked by \cite{recht2011hogwild}, although precursors exist. Later work in \cite{liu2015asynchronous,lian2015asynchronous} refined this analysis and generalised to nonconvex \emph{smooth} problems, under restricted scheduling models. 
Subsequently,~\cite{cannelli2016asynchronous} introduced a more general
and realistic probabilistic model of asynchronous computation on shared memory architectures, and analyzed an algorithm based on successive convex approximation and coordinate-wise updates. 

Asynchronous proximal gradient methods have been studied in \cite{zhu2018asynchronous,li2018simple} for problems of the form $f(x) + g(x)$, where $f(x)$ is smooth and nonconvex, and $g(x)$ is nonsmooth but with an easily computable closed form prox expression. In particular, expected rates of convergence were established. This class of problems is only relevant for training deep neural networks with neither ReLU activation functions nor max-pooling. That is, every activation function must be smooth. However, current popular deep neural network architectures make extensive use of nonsmooth activations, and thus all of the previous literature on convergence, speedup, etc. on asynchronous parallel SGD does not apply to these objectives, when considered from the standpoint of mathematical rigor. This leaves a clear gap between practical methods and their convergence guarantees.

This gap is understandable, given that the general problem of nonsmooth and nonconvex stochastic optimisation is notoriously difficult~\cite{bagirov2014introduction}.  A standard framework to establish convergence of stochastic algorithms in the centralized sequential setting is stochastic approximation, with early work in \cite{ermol1998stochastic,ruszczynski1987linearization} and comprehensive surveys in \cite{kushner2003stochastic} and \cite{borkar2008stochastic}. In \cite{davis2018stochastic,majewski2018analysis}, stochastic approximation for (sequential) nonsmooth and nonconvex problems has been recently developed, motivated by training deep neural networks.

In this paper, we take a first step towards bridging the gap and
establish the convergence of stochastic subgradient descent for nonsmooth and nonconvex problems in a realistic asynchronous computation framework. 
In particular, we show that generic asynchronous stochastic subgradient methods converge with probability one for a general class of nonsmooth and nonconvex problems. 
Aside from the limited discussion in~\cite[Chapter 12]{kushner2003stochastic}, this is the first result for this class of algorithms, combining the state of the art in stochastic approximation with that in asynchronous computation. 
In addition, inspired by the success of momentum methods~\cite{zhang2017yellowfin}, we also establish for the first time convergence of stochastic subgradient descent with momentum in the context of asynchronous computation.
%\begin{tabular}{p{0mm}p{16mm}p{3.5mm}p{3.5mm}p{3.5mm}p{3.5mm}}

%{%\setlength{\columnsep}{0pt}%
\begin{wrapfigure}{r}{8.1cm}
	\centering\small\vspace{-5pt}	
	\begin{tabular}{p{21mm}|p{12mm}p{6mm}|p{12mm}p{9mm}}
		\toprule
		Algorithm & Val top-1{\newline}Accuracy & Time{\newline}(Sec) & Val top-1{\newline}Accuracy & Time{\newline}(Sec)\\
		\midrule
		\multicolumn{1}{c}{}&\multicolumn{2}{c}{CIFAR10}& \multicolumn{2}{c}{CIFAR100}\\
		\midrule
		\asm & 92.44$\pm$0.3 &  1803 & 68.66$\pm$0.2 &  1807 \\
		\pasm & 91.03$\pm$0.3 &  992 & 61.26$\pm$0.5 &  1258 \\
		\rowcolor{LightCyan}
		{\pasmp} & \textbf{92.7$\pm$0.1} & \textbf{1526} & \textbf{69.28$\pm$0.3} & \textbf{1541}\\
		\sgd   & 92.76 & 2024 & 69.02 & 2022 \\
		\sgd (BS:1024)  &  92.57 & 1908 &  68.51 & 1762 \\
		\bottomrule
	\end{tabular}\caption{\small Resnet20 training for 300 epochs on a Nvidia GeForce RTX 2080 Ti. Batch-size of 256 was taken. For large-batch training, we follow \cite{goyal2017accurate}. Asynchronous training uses 4 concurrent processes. Standard hyperparameter values \cite{he2016deep} were applied.}
\end{wrapfigure}\label{tab:summary}

%From the practical perspective, there are (at least) two obstacles to using stochastic subgradient methods for training deep neural networks. 
%The first is technical, as these methods are not efficiently supported on current frameworks---in particular, the fact that the subgradient is computed only on a subset of the model parameters does not always yield computational savings. 
%The second concerns generalization, as such methods tend to yield worse validation/test accuracy in training (see Section~\ref{sec:exp} for a detailed discussion).  

We complement the convergence analysis with a new efficient implementation strategy. 
Specifically, our main convergence result applies to an \textbf{A}synchronous \textbf{S}tochastic \textbf{S}ubgradient \textbf{M}ethod (\textbf{\asm}), where each process updates all of the model parameters, and each partition is protected from concurrent updates by a lock. 
 A variation is to assign \emph{non-overlapping partitions} of the model to processes, which we call \textbf{P}artitioned \asm (\textbf{\pasm}). This prevents overwrites, and allows us to update the model in a lock-free manner. In practice, \asm updates the entire model, thus has an equivalent computation cost to the sequential minibatch \textbf{SGD}. By contrast, \pasm needs to compute block-partitioned stochastic subgradients. To implement this efficiently, we perform ``restricted'' backpropagation (see details in Section \ref{sec:exp}), which can provide savings proportional to the size of subgradients.  
 
 Although \pasm is fast, in practice, it does not always recover  validation (generalization) accuracy. To address this, we propose another specialization of \asm which  interleaves \pasm and \asm steps, which we call \textbf{\pasmp}. From a technical perspective, the novelty of \pasmp is to exploit concurrency to save on the cost of computation and synchronization without compromising  convergence and generalization performance. A sample of our performance results for our Pytorch-based implementation \cite{paszke2017automatic} is given in Figure \ref{tab:summary}. \pasmp matches the baseline in terms of generalization and yet provides on average ${\sim}1.35$x speed-up for identical minibatch size. At the same time, it achieves ${\sim}1.2$x speed-up against a large-batch training method~\cite{goyal2017accurate}, with better generalization. The method is applicable to both shared-memory (where multiple processes can be spawned inside the same GPU up to its computational saturation) as well as in the standard multi-GPU settings.

\section{Problem Formulation}

Consider the minimization problem
\begin{equation}\label{eq:prob}
\min_{x \in \mathbb{R}^n} f(x),
\end{equation}
where $f:\mathbb{R}^n\to \mathbb{R}$ is locally Lipschitz continuous (but potentially nonconvex and nonsmooth) and furthermore, it is computationally infeasible
to evaluate $f(x)$ or an element of the Clarke subdifferential $\partial f(x)$. 
%The problem in~\eqref{eq:prob} has many applications in machine learning, including the
%training of parameters in deep neural networks \cite{Bottou2018}. 

In the machine learning context, $f(x)$ corresponds to a loss function evaluated on a model with parameters denoted by $x \in \mathbb{R}^n$, dependant on input data $A\in\mathbb{R}^{n\times m}$ and target values $y\in\mathbb{R}^m$ of high dimension. That is, $f(x)=f(x;(A,y))$, where $x$ is a parameter to optimize with respect to a loss function $\ell:\mathbb{R}^m\times\mathbb{R}^m\to \mathbb{R}$. Neural network training is then achieved by miminizing the empirical risk, where $f$ admits the decomposition
\[
f(x)=\frac{1}{M}\sum_{i=1}^M \ell(m(x;A_i);y_i)
\]
and $\{(A_i,y_i)\}_{i=1}^M$ is a partition of $(A,y)$.

We are concerned with algorithms that solve the general problem in \eqref{eq:prob} in a distributed fashion over shared-memory; i.e.,
using multiple concurrent processes. Typically, a process uses a CPU core for computation over the CPU itself or to have a client connection with an accelerator such as a GPU. Hereinafter, we shall be using the terms core and process interchangeably. 

In particular, we focus on the general \emph{inconsistent read} scenario: before computation begins, each core $c \in \{1,\ldots,\overline{c}\}$ is allocated a block of variables $I^c \subset \{1,2,\ldots,n\}$, for which it is responsible to update. At each iteration the core modifies a block of variables $i^k$, chosen randomly among $I^c$. Immediately after core $c$ completes its $k$-th iteration, it updates the model parameters over shared memory. The shared memory allows concurrent-read-concurrent-write access. The shared-memory system also offers word-sized atomic \texttt{read} and \texttt{fetch-and-add (faa)} primitives. Processes use \texttt{faa} to update the components of the model. 

A lock on updating the shared memory is only placed when a core writes to it, and hence the process of reading may result in computations based on variable values that never existed in memory. Such a scenario arises when block $1$ is read by core $1$, then core $3$ updates block $2$ while core $1$ reads block $2$. In this case, core $1$ computes an update with the values in blocks $1$ and $2$ that are inconsistent with the local cache of core $3$. 

We index iterations based on when a core writes a new set of variable values into memory. Let $\mathbf{d}^{k_c}=\{d^{k_c}_1,...,d^{k_c}_n\}$ be the vector
of delays for each component of the variable used by core $c$ to evaluate a subgradient estimate, thus the $j$-th component of $x^{k_c} = (x_1^{d_1^{k_c}},\ldots,x_n^{d_n^{k_c}})$ that is used in the computation of the update at $k$ may be associated with a different delay than the $j'$-th component.
%; that is, the $j$-th component is given by $x^{d^{k_c}_j}_j$.

In this paper, we study stochastic approximation methods, of which the classic stochastic gradient descent forms a special case. Since $f$ in~\eqref{eq:prob} is in general nonsmooth and nonconvex, we exploit generalized subgradient methods. Denote by $\xi^{k_c}$ the mini-batches (i.e., a subset of the data $\{(A_i,y_i)\}_{i=1}^M$) used by core $c$ to compute an element of the subgradient $g_{i^k}(x^{k_c} ;\xi^{k_c})$. Consistent with standard empirical risk minimization methods \cite{Bottou2018}, the set of mini-batches $\xi^{k_c}$ is chosen uniformly at random from $(A,y)$, independently at each iteration. %By the central limit theorem, the error is asymptotically Gaussian as $M \rightarrow \infty$.

%\input{sys.tex}

%We consider a shared-memory system with $p$ processes concurrently and asynchronously performing computations on independent compute-cores. We interchangeably use the terms process and core. The memory allows concurrent-read-concurrent-write (CRCW)\footnote{The proposed method will work even if only a CREW access is available because of partitioning of variables.} access. The shared-memory system offers word-sized atomic \texttt{read} and \texttt{fetch-and-add (faa)} primitives. Processes use \texttt{faa} to update the components of the variable.

%\subsection{Algorithm Description}
Consider the general stochastic subgradient algorithm under asynchronous updating and momentum in Algorithm~\ref{alg:assdone}, presented from the perspective of the individual cores. We remark that stochastic subgradient methods with momentum have been widely utilized due to their improved performance \cite{mitliagkas2016asynchrony,zhang2017yellowfin}. In the update of the $k_c$-th iterate by core $c$, we let $0 < m < 1$ be the momentum constant, $\gamma^{k_c}$ is the step-size and $g_{i^k}(x^{k_c} ;\xi^{k_c})$ is an unbiased estimate of the Clarke subgradient at the point $x_{i^{k_c}}^{k_c}$. The variable $u_{i^{k_c}}^{k_c}$ is the weighted sum of subgradient estimates needed to introduce momentum.

%is defined by 
%\[
%\begin{array}{l}
%u_{i^{k_c}}^{k_c+1} = m u_{i^{k_c}}^{k_c}+g_{i^k}(x^{k_c} ;\xi^{k_c})\notag\\
%x^{k_c+1}_{i^{k_c}}=x^{k_c}_{i^{k_c}}-(1-m)\gamma^{k_c} u_{i^{k_c}} 
%\end{array}
%\]

We make the following assumptions on the delays and the stochastic subgradient estimates.
\begin{assumption}\label{as:probsdels}
	There exists a $\delta$ such that $d^{k_c}_j\le \delta$ for all $j$ and $k$. Thus each $d^{k_c}_j\in \mathcal{D}\triangleq \{0,...,\delta\}^n$.
\end{assumption}

\begin{assumption}\label{as:subg}
	The stochastic subgradient estimates $g(x,\xi)$ satisfy
	\begin{enumerate}
		\item[(i)] $\mathbb{E}_\xi\left[g(x;\xi)\right]\in \partial f(x)+\beta(x)$
		\item[(ii)] $\mathbb{E}_\xi\left[\text{dist}(g(x;\xi),\partial f(x))^2\right] \le \sigma^2$
		\item[(iii)] $\|g(x;\xi)\|\le B_g$ w.p.1.
	\end{enumerate}
\end{assumption}

The term $\beta(x)$ in (ii) of Assumption~\ref{as:subg} can be interpreted as a bias term, which can arise when the Clarke subgradient is not a singleton; e.g., when $f(x) = |x|,~x = 0$. In general, $\beta(x)$ is zero if $f(\cdot)$ is continuously differentiable at $x$.

\begin{algorithm}[H]
\caption{Asynchronous Stochastic Subgradient Method for an Individual Core}\label{alg:assdone}
\hspace*{-5cm}
\begin{algorithmic}[1]
\STATE \textbf{Input:} $x_0$, core $c$. 
\WHILE{Not converged}
\STATE Sample $i^{k_c}$ from the variables $I_c$ corresponding to core $c$. 
\STATE Sample $\xi^{k_c}$ from the data $(A,y)$.
\STATE Read $x_{i^{k_c}}^{k_c}$ from the shared memory.
\STATE Compute the subgradient estimate $g_{i^k}(x^{k_c} ;\xi^{k_c})$. 
\STATE Write to the momentum vector 
$u^{k_c+1}_{i^{k_c}} \leftarrow m u^{k_c}_{i^{k_c}}+g_{i^k}(x^{k_c};\xi^{k_c})$ in private memory.
\STATE Write, with a lock, to the shared memory vector partition 
$x_{i^{k_c}}^{k_c+1} \leftarrow x_{i^{k_c}}^{k_c}-(1-m)\gamma^{k_c}u_{i^{k_c}}$.
\STATE $k_c \leftarrow k_c+1$.
\ENDWHILE
\end{algorithmic}
\end{algorithm}

\subsection{Continuous-Time Reformulation}\label{sec:cont_time}

Any realistic asynchronous system operates in continuous-time, and furthermore the Stochastic Approximation framework relies on discrete iterates converging to some continuous flow. As such, we now sketch the key assumptions on the continuous-time delays of core updates. Consider Algorithm \ref{alg:assdone}, dropping the core label $k_c$, considering a particular block $i$ and using a global iteration counter $k$. More details on the reformulation are given in Appendix Section \ref{app:subsec:reformulation}. We write,
\begin{align}\label{eq:eachitersa}
x_{i^{k+1}}^{k+1} = x_{i^k}^{k+1} + (1-m)\gamma^{k,i} \sum_{j=1}^k m^{k-j} Y_{j,i},
\end{align}
where $Y_{j,i}$ is based on an estimate of the partial subgradient with respect to block variables indexed by $i$ at local iteration $j$. 

In the context of Algorithm~\ref{alg:assdone}, the step size is defined to be the subsequence
$\{\gamma^{k,i}\}=\{\gamma^{\nu(c(i),l)}:i=i^l\}$ where $l$ is the iteration index
for the core corresponding to block $i$. That is, the step size sequence forms the subsequence of $\gamma^k$ for which $i^k=i$ is the block of variables being modified. 

The term $Y_{k,i}$ corresponds to $g(x_{k},\xi)$ and satisfies, 
\[
Y_{k,i} = g_{i} ((x_{k-[d^k_{i}]_1,1},...,x_{k-[d^k_i]_j,j},...,x_{k-[d^k_i]_n,n}))+\beta_{k,i}+\delta M_{k,i},
\]
where $g_{i}(x)$ denotes a selection of an element of the subgradient, with respect to block $i$, of $f(x)$. The quantity $\delta M_{k,i}$ is a martingale difference sequence.

In order to translate the discrete-time updates into real-time updates, we now consider an interpolation of the updates. This is a standard approach \cite{kushner2003stochastic}, which provides a means of establishing that the sequence of iterates converges to the flow of a differential inclusion. 

Define $\delta \tau_{k,i}$ to be the real elapsed time between iterations $k$ and $k+1$ for block $i$. Let $T_{k,i} = \sum_{j=0}^{k-1} \delta \tau_{j,i}$ and define for $\sigma\ge 0$, $p_l(\sigma) = \min\{j: T_{j,i}\ge \sigma\}$ to be the first iteration at or after $\sigma$. We assume that the step-size sequence comes from an underlying real function; i.e.,
\[
\gamma^{k,i} = \frac{1}{\delta \tau_{k,i}}\int_{T_{k,i}}^{T_{k,i}+\delta \tau_{k,i}} \gamma(s) \mathrm{d}s,\,\,\,\,\text{satisfying,}
\]
\begin{equation}\label{eq:stepsizeas}
\begin{array}{l}
\int_0^\infty \gamma(s)\mathrm{d}s=\infty,\text{ where }0<\gamma(s)\to 0\text{ as }s\to\infty,\\
\text{There are }T(s)\to \infty\text{ as }s\to \infty\text{ such that }
 \lim_{s\to\infty}\sup_{0\le t\le T(s)}\left|\frac{\gamma(s)}{\gamma(s+t)}-1\right|=0
\end{array}
\end{equation}

We now define two $\sigma$-algebras $\mathcal{F}_{k,i}$ and $\mathcal{F}^+_{k,i}$, which measure the random variables 
\[
\begin{array}{l}
\left\{\{x_0\},\{Y_{j-1,i}:j,i \text{ with } T_{j,i}<T_{k+1,i}\},\{T_{j,i}:j,i\text{ with }T_{j,i}\le T_{k+1,i}\}\right\},\text{ and,} \\
\left\{\{x_0\},\{Y_{j-1,i}:j,i \text{ with } T_{j,i}\le T_{k+1,i}\},\{T_{j,i}:j,i\text{ with }T_{j,i}\le T_{k+1,i}\}\right\},
\end{array}
\] 
indicating the set of events up to, and up to and including the computed noisy update at $k$, respectively. %Note that each of these constructions remains consistent with a core updating different blocks at random, with $\delta \tau_{k,i}$ arising from an underlying distribution for $\delta \tau_{k,c(i)}$. 

For any sequence $Z_{k,i}$, we write $Z^\sigma_{k,i}=Z_{p_i(\sigma)+k,i}$, where $p_i(\sigma)$
is the least integer greater than or equal to $\sigma$. Thus, let $\delta \tau_{k,i}^\sigma$ denote the inter-update times for block $i$ starting at the first update at or after $\sigma$, and $\gamma_{k,i}^\sigma$ the associated step sizes. Now let $x^\sigma_{0,i}=x_{p_i(\sigma),i}$ and for $k\ge 0$, $x^\sigma_{k+1,i}=x^\sigma_{k,i}+(1-m)\gamma^\sigma_{k,i} \sum_{j=1}^k m^{k-j} Y^\sigma_{j,i}$. Define $t^\sigma_{k,i}=\sum_{j=0}^{k-1} \gamma^\sigma_{j,i}$ and $\tau^\sigma_{k,i}=\sum_{j=0}^{k-1} \gamma^\sigma_{j,i} \delta \tau^\sigma_{j,i}$. Piece-wise constant interpolations of the vectors in real-time are then given by
\[
x^\sigma_i(t)=x^\sigma_{k,i},\quad t\in[t^\sigma_{k,i},t^\sigma_{k+1,i}],\quad
%\hat x^\sigma_i(t)=\hat x^\sigma_{k,i},\quad 
N^\sigma_{i}(t)=t^\sigma_{k,i},\quad t\in[\tau^\sigma_{k,i},\tau^\sigma_{k+1,i})
\]
and $\tau^\sigma_i(t)=\tau^\sigma_{k,i}$ for $t\in[t^\sigma_{k,i},t^\sigma_{k+1,i}]$. We also define,
\begin{align}
    \hat x^\sigma_i(t)=x^\sigma_i(N^\sigma_i(t)),~~~t \in [\tau^{\sigma}_{k,i},\tau^{\sigma}_{k+1,i}).
\end{align}
Note that
\[
%\begin{array}{l}
N^\sigma_i(\tau^\sigma_i(t))=t^\sigma_{k,i},\,  t\in[t^\sigma_{k,i},t^\sigma_{k+1,i}].
%\end{array}
\]
%x^\sigma_i(t)=\hat x^\sigma_i(\tau^\sigma_i(t)),\,

We now detail the assumptions on the real delay times, which ensure that the real-time delays do not grow without bound, either on average, or with substantial probability. %Intuitively, this means that it is highly unlikely that any core decelerates exponentially in its computation speed. We denote expectations conditioned on $\mathcal{F}_{k,i}$ by $\mathbb{E}_{k,i}[\cdot]$ and expectations conditioned on $\mathcal{F}_{k,i}^+$ by $\mathbb{E}_{k,i}^+[\cdot]$. Furthermore, $\mathbb{E}_{k,i}^{\sigma}[\cdot] = \mathbb{E}_{p_i(\sigma) + k,i}[\cdot]$.

\begin{assumption}\label{as:uniformintdt}
$\{\delta \tau^\sigma_{k,i};k,i\}$ is uniformly integrable.
\end{assumption}

\begin{assumption}\label{as:delaymean} 
There exists a function $u^\sigma_{k+1,i}$ and random variables $\Delta^{\sigma,+}_{k+1,i}$ and a random sequence
$\{\psi^\sigma_{k+1,i}\}$ such that 
\[
\mathbb{E}^{+}_{k,i} [\delta \tau^\sigma_{k+1,i}] = u^\sigma_{k+1,i} (\hat x_i^\sigma(\tau^\sigma_{k+1,i}-\Delta^{\sigma,+}_{k+1,i}),\psi^\sigma_{k+1,i})
\]
and 
there is a $\bar u$ such that for any compact set $A$,
\[
\lim_{m,k,\sigma}\frac{1}{m} \sum_{j=k}^{k+m-1}\mathbb{E}^\sigma_{k,i}[u^\sigma_{j,i}(x,\psi^\sigma_{k+1,i})-\bar u_i(x)] I_{\{\psi^\sigma_{k+1,i}\in A\}}=0
\]
\end{assumption}
\begin{assumption}\label{as:bias}
	\begin{align}\label{eq:cond2}
	\lim_{m,k,\sigma} \frac{1}{m} \sum_{j = k}^{k+m-1} \mathbb{E}_{k,i}^{\sigma}[\beta_{j,i}^{\sigma}]  = 0~\text{in mean.}
	\end{align}
	\end{assumption}

Under the other assumptions in this section, we note that Assumption~\ref{as:bias} holds if the set of $x$ for which $f(\cdot)$ is not continuously differentiable at $x$ is of measure zero. This is the case for objectives arising from a wide range of deep neural network architectures \cite{davis2018stochastic}.

\section{Main Convergence Results}

We now present our main convergence result. The proof is available in  Appendix Section \ref{app:sec:proof}.
\begin{theorem}\label{thrm:main_convergence}
Suppose Assumptions~\ref{as:probsdels}, \ref{as:subg}, \ref{as:uniformintdt}, \ref{as:delaymean}, \ref{as:bias} hold.%~\ref{as:probsdels},~\ref{as:subg},~\ref{as:uniformintdt}, and~\ref{as:delaymean}. 

Then, the following system of differential inclusions,
\begin{equation}\label{eq:sol}
\tau_i(t)=\int_0^t \bar u_i(\hat x(\tau_i(s))) \mathrm{d}s, \quad %,\label{eq:taufinal}\\
\dot x_i(t)\in \partial_i f(\hat x(\tau_i(t))),\quad %\label{eq:diprob1}\\
\dot{\hat{x}}_i(t) \bar u_i(\hat x) \in \partial_i f(\hat x(t))% \label{eq:diprob2}
\end{equation}
holds for any $\overline{u}$ satisfying~\ref{as:delaymean}.
On large intervals $[0,T]$, $\hat x^\sigma(\cdot)$ spends nearly all of its time, with the fraction going to one as $T\to \infty$ and $\sigma\to \infty$ in a small neighborhood of a bounded invariant set of
\begin{equation}\label{eq:di}
\dot{\hat{x}}_i(t)\in\partial_i f(x(t))
\end{equation}
\end{theorem}

%\subsection{Convergence with Probability One}\label{sec:prob_1}

Theorem~\ref{thrm:main_convergence} provides conditions under which the time-shifted interpolated sequence of iterates converges weakly to an invariant set of a DI. This result can be strengthened to convergence with probability one to a block-wise stationary point via modification of the methods in \cite{dupuis1989stochastic}. Details of the proof are available in Appendix Section \ref{app:subsec:convproof} along with further discussion on properties of the limit point in Appendix Section \ref{app:subsec:limpoint}. 

We now discuss this result. Consider three variations of stochastic subgradient methods discussed in the introduction: 1) standard sequential \sgd with momentum, 2) \textbf{PASSM}, where each core updates only a block subset $i$ of $x$ and is lock-free, 
and 3) \textbf{ASSM}, which defines the standard parallel asynchronous implementation in which every core updates the entire vector $x$, using a write lock to prevent simultaneous modification of the vectors. 

\sgd only has asymptotic convergence theory (i.e., as opposed to convergence rates) to stationary points for general nonconvex nonsmooth objectives, in particular those characterizing deep neural network architectures. All previous works considering asynchronous convergence for such objectives  assume that the objective is either continuously differentiable, or the nonsmooth term is simple, e.g., an added $\ell_1$ regularization, and thus not truly modeling DNN architectures. While these results shed insight into variations of SGD, the focus of this paper is to attempt to expand the state of the art to realistic DNNs under complete mathematical rigor. 

By simply taking the block $i$ to be the entire vector, \textbf{ASSM} can be taken to a special case of \textbf{PASSM}, and thus Theorem~\ref{thrm:main_convergence} proves asymptotic convergence for both schemes. The Theorem points to their comparative advantages and disadvantages: 1) in order for indeed Theorem~\ref{thrm:main_convergence} to apply to \textbf{ASSM}, write locks are necessary, thus limiting the potential time-to-epoch speedup, however, 2) whereas if $i$ is the entire vector the limit point of \textbf{ASSM} is a stationary point, i.e., a point wherein zero is in the Clarke subdifferential of the limit, in the case of \textbf{PASSM}, the limit point is only coordinate-wise stationary, zero is only in the $i$ component subdifferential of $f$. In the smooth case, these are identical, however, this is not necessarily the case in the nonsmooth case. Thus \textbf{PASSM}, relative to \textbf{ASSM} can exhibit faster speedup allowing a greater advantage of HPC hardware, however, may converge to a weaker notion of stationarity and thus in practice a higher value of the objective. 

% These distinctions will be seen in the Numerical Results below, and also suggest the novel implementation strategy we present as \textbf{PASSM+}, wherein we alternate between iterations of \textbf{PASSM} and \textbf{ASSM} in order to attempt to take advantage of the positive features of each approach.

\section{Numerical Results}\label{sec:exp}

In this section we describe the implementation and experimental results of the specialized variants, called \asm, \pasm, and \pasmp, of the shared-memory-based distributed Algorithm \ref{alg:assdone}.

\paragraph{Experiment Setting.} Our implementations use the Pytorch library \cite{paszke2017automatic} and the multi-processing framework of Python.
We implemented the presented methods on two settings: (\textbf{S1}) a machine with a single Nvidia GeForce GTX 2080 Ti GPU, and (\textbf{S2}) a machine with 10 Nvidia GeForce GTX 1080 Ti GPUs on the same board. These GPUs allow concurrent launch of CUDA kernels by multiple processes using Nvidia's multi-process service (MPS). MPS utilizes Hyper-Q capabilities of the Turing architecture based GPUs GeForce GTX 1080/2080 Ti. It allocates parallel compute resources -- the streaming multiprocessors (SMs) -- to concurrent CPU-GPU connections based on their availability. The system CPUs (detail in Appendix Section \ref{app:sec:experiment}) have enough cores to support non-blocking data transfer between main-memory and GPUs via independent CPU threads. 

\paragraph{Scalable Asynchronous Shared-memory Implementation.} On a single GeForce GTX 2080 Ti, we can allocate up to 11 GB of memory on device, whereas, GeForce GTX 1080 Ti allows allocation of up to 12 GB on device. Typically, the memory footprint of Resnet20 \cite{he2016deep} over CIFAR10 \cite{krizhevsky2009learning} for the standard minibatch-size (BS) of 256 is roughly 1.1 GB. However, as the memory footprint of a training instance grows, for example, DenseNet121 \cite{HuangLMW17} on CIFAR10 requires around 6.2 GB on device for training with a minibatch size of 128,  which can prevent single GPU based concurrent or large-batch training. Notwithstanding, a multi-GPU implementation provides a basis for scalability. 

Parameter server (PS) \cite{LiAPSAJLSS14} is an immediate approach for data-parallel synchronous distributed SGD. 
Alternatively, the  decentralized approaches are becoming increasingly popular for their superior scalability.
%In this case, each of the GPUs have access to the entire dataset and they independently compute subgradients for a minibatch and perform synchronous all-reduce thereof to update the model. 
Pytorch provides a state-of-the-art implementation of this method in the \texttt{nn.parallel.DistributedDataParallel()} module.
For our purposes, we implement concurrent parameter server as described below:
\begin{itemize}[leftmargin=0.1cm,align=left,labelwidth=\parindent,labelsep=4pt,nosep]
	\item A single GPU hosts multiple concurrent processes working as asynchronous PSs. Each PS can have multiple slaves hosted on independent GPUs.  More specifically, with an availability of $n$ GPUs, we can spawn $c$ PSs each having $s$ slaves as long as $cs+1\le n$. For instance, in our setting \textbf{S2}, we spawn 3 PSs on the first GPU and assign 3 slaves to each, thus populating the array of 10 GPUs.
	\item The PSs and their slaves have access to the entire dataset. For each minibatch they decide on a common permutation of length as the sum of their assigned minibatch sizes and partition the permutation to draw non-overlapping samples from the dataset. They shuffle the permutation on a new epoch (a full pass over the entire dataset). This ensures I.I.D. access to the dataset by each PS and its slaves. Typically, the minibatch size for a PS is smaller than that of its slaves to ensure load-balancing over the shared GPU hosting multiple concurrent PSs. 
\end{itemize}
Notice that our implementation has a reduced synchronization overhead at the master GPU in comparison to a PS based synchronous training. Furthermore, as a byproduct of this implementation, we are able to train with a much smaller "large" minibatch size compared to a synchronous decentralized or PS implementation, which faces issues with generalization performance \cite{GoyalDGNWKTJH17}. In summary, our implementation applies to two experimental settings:
\begin{enumerate}[leftmargin=0.1cm,align=left,labelwidth=\parindent,labelsep=4pt, label=(\alph*),nosep]
\item \textbf{A single GPU} is sufficient for the training: instead of increasing the baseline minibatch size, use multiple concurrent processes with the same minibatch size.
\item A single GPU can not process a larger minibatch and therefore \textbf{multiple GPUs} are necessary for scalability: instead of a single PS synchronizing across all the GPUs, use concurrent PSs synchronizing only subsets thereof.
\end{enumerate}

\paragraph{Block Partitioned Subgradients.} 
Subgradient computation for a NN model via backpropagation is provided by the autograd module of Pytorch. Having generated a computation graph during the forward pass, we can specify the weights and biases of the layers along which we need to generate a subgradient in a call to
 \texttt{torch.autograd.grad()}. We utilize this functionality in order to implement "restricted" backpropagation in \pasm. Notice that, computing the subgradient with respect to the farthest layer from the output incurs the cost of a full backpropagation. We assign the layers to concurrent processes according to roughly equal \textit{division of weights}. With this approach, on average we save $F\frac{(c-1)(c-2)}{c}$ floating point operations (flops) in subgradient computation with $c$ processes in the system, where $F$ is the flops requirement for backpropagation on one minibatch processing iteration. At a high level, our subgradient partitioning may give semblance to model partitioning, however note that, most commonly model partitioning relies on pipeline parallelism \cite{abs-1806-03377} , whereas we compute subgradients as well as perform forward pass independently and concurrently. 

\paragraph{\pasmp Heuristic.} As described above,  \pasm effectively saves on computation, which we observe in practice as well, see Figure \ref{tab:summary}. However, as discussed in the analysis, \pasm converges to a block-wise stationary point as opposed to a potential stationary point that \asm or \sgd would converge to, after processing an identical amount of  samples. Obviously we aim to avoid compromising on the quality of optimization. 
%With that aim, we investigate the optimization trajectory of \asm/\sgd. 
\begin{figure*}[t]
\centering
\adjustbox{valign=b}{\subfigure[Performance Summary.]{%
\scriptsize\vspace{0pt}
\begin{tabular}{p{0mm}p{17mm}|p{3.mm}p{3.5mm}p{3.5mm}p{3.5mm}p{3.5mm}}
	\toprule\scriptsize
	{} & Algo & Train{\newline}Loss & Train{\newline}Acc@1 & Test{\newline}Loss & Test{\newline}Acc@1 & Time{\newline}(Sec) \\
	\midrule
(1) &           \sgd &     0.012 &    99.91 &    1.405 &       68.57 &     4262 \\
(2) &          \asm &     0.010 &    99.95 &    1.390 &       68.12 &     3545 \\
(3) &         \pasm &     1.146 &      69.88 &    1.414 &     61.40 &     1773 \\\rowcolor{LightCyan}
(4) &      \pasmp &     0.030 &    99.63 &    1.224 &       \textbf{68.59} &     \textbf{2668} \\
(5) &           \sgd (BS=1024) &     0.006 &  99.97 &    1.533 &       68.15 &     3179 \\
\bottomrule
\end{tabular}

}}
\hfill
\adjustbox{valign=b}{\subfigure[Train Loss.]{%
\includegraphics[width=0.25\textwidth]{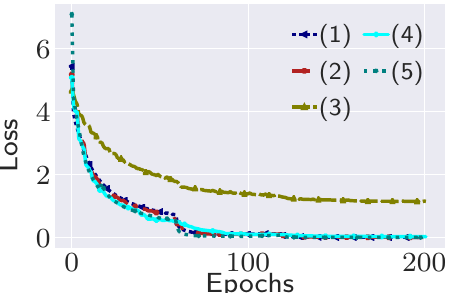}
}}
\hfill
\adjustbox{valign=b}{\subfigure[Top1 Val Accuracy.]{%
	\includegraphics[width=0.25\textwidth]{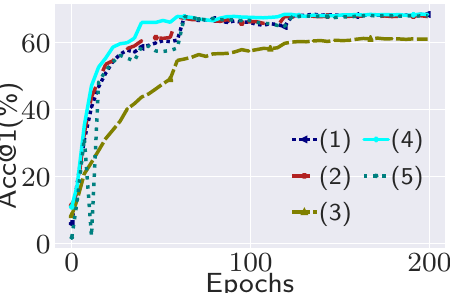}
}}\caption{\small First, we take the case of a relatively small architecture, Shufflenet \cite{ZhangZLS18} with 1012108 parameters, training over CIFAR100  on the setting \textbf{S1}. The memory footprint for baseline SGD with BS=128 is $\sim$2.1 GB. The asynchronous methods use BS=128 and train with 4 processes. The large-batch training gets LR warm-up. The initial LR, weight-decay, momentum are identical across the methods. \pasmp provides 1.2x speed-up compared to the large-batch method and 1.6x compared to the baseline with a superior validation accuracy.}\label{fig:shufflenet-cifar100}
\end{figure*}
\begin{figure*}[t]
\centering
\adjustbox{valign=b}{\subfigure[Performance Summary.]{%
\scriptsize\vspace{0pt}
\begin{tabular}{p{0mm}p{16mm}|p{3.5mm}p{3.5mm}p{3.5mm}p{3.5mm}p{3.5mm}}
	\toprule\scriptsize
	{} & Algo & Train{\newline}Loss & Train{\newline}Acc@1 & Test{\newline}Loss & Test{\newline}Acc@1 & Time{\newline}(Sec) \\
	\midrule
	(1) &          \asm &     0.001 &     99.99 &    0.287 &     93.45 &    17570 \\
	(2) &         \pasm &     0.383 &      92.75 &    0.391 &     87.41 &     9659 \\\rowcolor{LightCyan}
	(3) &       \pasmp &     0.002 &      99.99 &    0.324 &     \textbf{94.07} &    \textbf{13047} \\
	(4) &           \sgd &     0.002 &      99.99 &    0.273 &     94.36 &    16199 \\
	(5) &           \sgd (BS=512) &     0.001 &     99.99 &    0.297 &     93.93 &    15314 \\
	\bottomrule
\end{tabular}

}}
\hfill
\adjustbox{valign=b}{\subfigure[Train Loss.]{%
\includegraphics[width=0.25\textwidth]{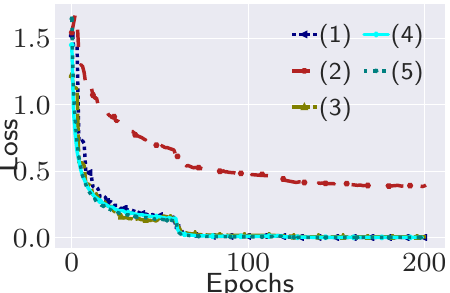}
}}
\hfill
\adjustbox{valign=b}{\subfigure[Top1 Val Accuracy.]{%
	\includegraphics[width=0.25\textwidth]{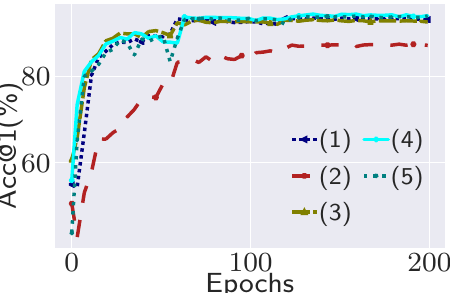}\vspace{-2pt}
}}\caption{\small Now, we consider WideResnet16x8 \cite{ZhangZLS18} with 11012036 parameters training over CIFAR10 on the setting \textbf{S1}. The initial LR, weight-decay, momentum are identical across the methods. The memory footprint for SGD with BS=64 is $\sim$2 GB. Here we compute subgradients with BS=64 and update the model at BS=128 for the baseline SGD and the asynchronous methods, which spawn 4 concurrent processes for training. The large-batch SGD computes subgradients with BS=256, and updates the model with BS=512, and is given LR warm-up. Notice that as the model size grows, we face memory constraints and higher training time, which clearly needs scaling. In this set of experiments, \pasmp achieves 1.25x speed-up in comparison to baseline and 1.17x in comparison to a large-batch method, with a generalization performance at par.}\label{fig:wnet168-cifar10}
\end{figure*}
\begin{figure*}[t]
\centering
\adjustbox{valign=b}{\subfigure[Performance Summary.]{%
\scriptsize\vspace{0pt}
\begin{tabular}{p{11mm}p{6mm}|p{3.mm}p{3.5mm}p{3.5mm}p{3.5mm}p{3.5mm}}
	\toprule\scriptsize
	Algo& Dataset & Train{\newline}Loss & Train{\newline}Acc@1 & Test{\newline}Loss & Test{\newline}Acc@1 & Time{\newline}(Sec) \\
	\midrule
(1)~\sgd & C100&    0.005 &      99.97 &  1.355 &       74.31 &     4478 \\
(2)~\sgd & C10&    0.002 &     99.99 &     0.290 &     93.35 &     3751 \\\rowcolor{LightCyan}
(3)~\pasmp &  C100&   0.004 &     100.00 &    1.402 &     \textbf{74.67} &     \textbf{3827} \\\rowcolor{LightCyan}
(4)~\pasmp & C10&    0.002 &     99.99 &    0.302 &      \textbf{93.33} &     \textbf{3307} \\
\bottomrule
\end{tabular}

}}
\hfill
\adjustbox{valign=b}{\subfigure[Train Loss.]{%
\includegraphics[width=0.25\textwidth]{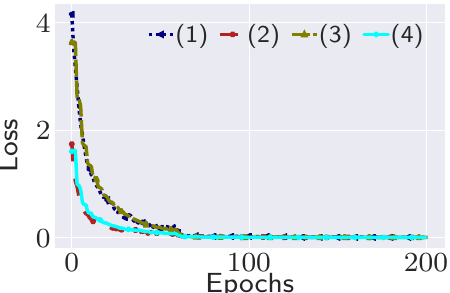}
}}
\hfill
\adjustbox{valign=b}{\subfigure[Top1 Val Accuracy.]{%
	\includegraphics[width=0.25\textwidth]{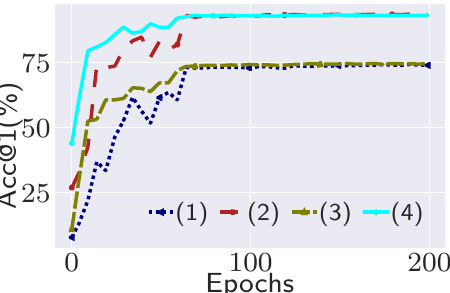}
}}\caption{\small Finally we consider an architecture with even bigger memory requirement, Densenet121 \cite{HuangLMW17} with 7048548 parameters, training over CIFAR10/100 (C10/100). Its memory footprint of $\sim$6.2 GB with BS=128 for \sgd makes it practical to train on the distributed setting \textbf{S2}. Here \sgd is decentralized over 10 GPUs and is implemented by a state-of-the-art (SOTA) framework \texttt{DistributedDataParallel} of Pytorch. \pasmp spawns 3 concurrent PSs. Here \sgd is a large-batch implementation, which computes subgradients at BS=128 and updates the model at BS=1280 on aggregation. On the other hand, the slaves of each PS compute subgradient at BS=128 and the PSs themselves compute subgradient at BS=32 in order for load-balancing at the shared GPU. We do LR warm-up for both the methods. Because the updates by PSs and \sgd have different BSs, we use different LR for them, essentially in the ratio of their BSs: $\frac{416}{1280}$.  The weight-decay and momentum are identical across the methods. Compared to the SOTA implementation, \pasmp provides speed-up of 1.17x for CIFAR100 and 1.13x for CIFAR10 with matching validation accuracy. We can explain the speed-up in terms of (a) reduced flops during backpropagation, (b) reduced communication cost for partitioned subgradients across GPUs, and (c) potentially reduced synchronization cost in a distributed setting.}\label{fig:densenet}
\end{figure*}

% Notice that, on average, at each iteration aggregating over 10 GPUs, \pasmp has a slightly smaller BS (1248) compared to \sgd (1280). 

During the course of an end-to-end training, a standard diminishing learning rate (LR) scheme is the following: the LR is dampened by a constant factor $\gamma$ when a fixed portion of the sample-processing budget is consumed. A common practice is to dampen LR when 50\% and 75\% samples are processed. Some practitioners also prefer dampening LR when 30\%, 60\% and 80\% samples are processed \cite{GoyalDGNWKTJH17}. 

Now, an immediate clue to improve the optimization by \pasm would be to spend some floating point operations saved by partitioned subgradient computation. We observed that during the initial phase of training, doing full subgradient updates helps. However, on further exploration we noticed that even if we perform full subgradient updates during the first 50\% number of epochs, it was insufficient for \pasm to recover the same level of train and validation accuracy. Furthermore, we noticed that doing the full subgradient updates during a small number of epochs around LR dampening steps would improve the convergence results. 

Based on these observations, we developed the following heuristic. We perform roughly 25\% \asm iterations initially. After that we switch to \pasm iterations. Further down the course of optimization, we throw in roughly 10\% \asm iterations around the LR dampening phase, thus in total performing roughly 50\% \asm and 50\% \pasm iterations. On top of this, we marginally dampen the LR by a factor of $(1-1/c)$, where $c>1$ is the number of concurrent processes, when switching from \asm to \pasm. Our experiments show that this heuristic was good enough to almost always recover the optimization quality that full subgradient computation methods would achieve. 

\paragraph{Experimental Observations, Discussions and Summary.}
We present detailed observations and discussions thereof for a number of well-known CNN architectures for image classification in Figures \ref{fig:shufflenet-cifar100}, \ref{fig:wnet168-cifar10}, and \ref{fig:densenet}. Each of the experimental runs is seeded, however, there is always a system-dependent randomization in asynchronous updates due to multiprocessing. Therefore, we take the average of three runs. 

Hyperparameter tuning for optimization of neural networks is a challenging task for practitioners \cite{teja1910}. Many existing work particularly suggest tuning other hyperparameters such as momentum \cite{SutskeverMDH13,zhang2017yellowfin}. Yet others tune the learning rate adapted to the performance on dev set \cite{DuchiHS11}. In comparison to these methods, our heuristic is simple enough, yet we believe that on further tuning the hyperparameters \pasmp would be an extremely efficient method that would save on computation with a quality solution for neural network training. In the experiments, we followed the well-accepted strategies with respect to the system components, such as number of CPU threads to load data, to run the baseline; it is especially encouraging to see that \pasmp outperforms SOTA on this count.

\section{Conclusion}

In this paper we analyzed the convergence theory of asynchronous parallel stochastic subgradient descent. Under a realistic probabilistic model of asynchronous parallel architecture applied to the stochastic subgradient method with momentum, we showed asymptotic convergence to a stationary point of general nonconvex nonsmooth objectives.

We presented numerical results that indicate that the theory suggests an efficient implementation strategy for this class of optimization methods to train deep neural networks.
%\section*{Broader Impact}
%As this work is theoretical and algorithmic, its impact is too broad to make definitive statements about its ethical and societal impacts. However, insofar as the use of deep neural networks in learning, and high performance computing for training, is expected to increase and itself have wide ranging impact, largely positive on technological development, this paper, in enhancing their understanding should contribute to this impact.
%

\bibliographystyle{plain}
\bibliography{refs}
\appendix

\section{The Probabilistic Computation Model}

\subsection{Motivation}

Classical convergence analysis for parallel asynchronous optimization using shared memory has been based on a simplistic model of asynchronous computation. In particular, it is assumed that every block of the parameter vector has an equal probability of being updated at every iteration. Moreover, the delay between the the age of each block at each core and the shared memory is fixed. As a consequence, each core reads and updates each block in a symmetric fashion.

The simplistic model of asynchronous computation in \cite{zhu2018asynchronous,li2018simple} is also not consistent with standard practice in real processing architectures. For example, in the common Non-Uniform Memory Access (NUMA) architecture, experiments have shown that it can be more effective for each core to update only a subset of blocks. As such, the block that is updated at each iteration depends on the previous iterates and the last core to update. This induces a probabilistic dependence between delays and the block to be updated. 

To account for these realistic features of asynchronous computation, we consider the model given in \cite{cannelli2016asynchronous} and apply it to the stochastic approximation scheme. In this work, an algorithm based on successive convex approximation and coordinate-wise updates was also proposed and analysed. 

\subsection{Algorithm from the ``Global" Perspective}

We first reformulate Algorithm 1 from the perspective of a global counter, indicating sequential updates of any variable block by any core. In iteration $k$, the updated iterate $x_{i^k}^{k+1}$ depends on a random vector $\zeta^k \overset{\triangle}{=} (i^k,d^k,\xi^k)$, where $i^k$ is the variable block to be updated, while $d^k$ and $\xi^k$ are the vector of delays for the variables and the set of mini-batches, respectively, for the core performing the update. With this notation, Algorithm 1 in the main text can be rewritten from the global perspective as given in Algorithm~\ref{alg:assd} here.

\begin{algorithm}[H]
	\caption{Asynchronous Stochastic Subgradient Method with a Global Counter}\label{alg:assd} 
	\begin{algorithmic}[1]
		\STATE \textbf{Input:} $x_0$.
		\FOR{$k = 1,2,\ldots$}
		\STATE Sample $\zeta^k=(i^k,d^k,\xi^k)$.
		\STATE Update $u_{i^k} \leftarrow m u_{i^k} + g_{i^k}(x^{k},\xi^k)$
		\STATE Update $x^{k+1}_{i^k} \leftarrow x^k_{i^k}-(1-m)\gamma^{\nu(k)} u_{i^k}$
		\STATE $k \leftarrow k+1$
		\ENDFOR
	\end{algorithmic}
\end{algorithm}

%\begin{remark}
%	In \cite{cannelli2016asynchronous}, additional assumptions are also placed on the time between updates (in discrete-time) in order to ensure that each core does not have arbitrarily large delays in the time to perform updates. In the context of stochastic gradient algorithms, it is useful instead to place conditions on the real-time delays between updates. The relevant assumption is provided in Section~\ref{sec:cont_time}. 
%\end{remark}

In order to obtain global convergence, we require a diminishing step size. However, as synchronization is not feasible, each core does not necessarily have access to a global counter to inform the step size decrease. Instead, each core has its own local step size $\gamma^{\nu(c^k,k)}$, where $c^k$ is the core in the $k$-th global iteration. We also define the random variable $Z^k$ as the element of $\{1,\ldots,\overline{c}\}$ that is active in global iteration $k$. The random variable $\nu(c^k,k) \overset{\triangle}{=} \sum_{j=0}^k I(Z^j = c^k)$ denotes the number of updates performed by core $c^k$ before the $k$-th iteration. 

In practice, it has been observed that it is more efficient to partition variable blocks over the cores, rather than allowing every processor to update any variable block \cite{liu2015asynchronous}. Using this approach, $c^k$ is uniquely defined by $i^k$, the block variable index updated in iteration $k$.

\subsection{Continuous-Time Reformulation of the Probabilistic Model}\label{app:subsec:reformulation}

Recall the update mechanism for the Algorithm
\begin{align}\label{eq:eachitersa}
x_{i^{k+1}}^{k+1} = x_{i^k}^{k+1} + (1-m)\gamma^{k,i} \sum_{j=1}^k m^{k-j} Y_{j,i},
\end{align}
where $Y_{j,i}$ is based on an estimate of the partial subgradient with respect to block variables indexed by $i$ at local iteration $j$. 

In the context of Algorithm 1, the step size is defined to be the subsequence
$\{\gamma^{k,i}\}=\{\gamma^{\nu(c(i),l)}:i=i^l\}$ where $l$ is the iteration index
for the core corresponding to block $i$. That is, the step size sequence forms the subsequence of $\gamma^k$ for which $i^k=i$ is the block of variables being modified. 

The term $Y_{k,i}$ corresponds to $g(x_{k},\xi)$ and satisfies, 
\[
Y_{k,i} = g_{i} ((x_{k-[d^k_{i}]_1,1},...,x_{k-[d^k_i]_j,j},...,x_{k-[d^k_i]_n,n}))+\beta_{k,i}+\delta M_{k,i},
\]
where $g_{i}(x)$ denotes a selection of an element of the subgradient, with respect to block $i$, of $f(x)$. The quantity $\delta M_{k,i}$ is a martingale difference sequence,
satisfying $\delta M_{k,i} = M_{k+1,i}-M_{k,i}$ for a martingale $M_k$, a sequence of random variables which satisfies $\mathbb{E}[M_{k,i}]<\infty$ and $\mathbb{E}[M_{k+1,i}|M_{j,i},j\le k] =M_{k,i}$ with probability $1$ for all $k$. It then holds that $\mathbb{E}[|M_{k,i}|^2]<\infty$ and $\mathbb{E}[M_{k+1,i}-M_{k,i}][M_{j+1,i}-M_{j,i}]'=0$. As a consequence, $\mathbb{E}_{k,i}[\delta M_{k,i}]=0$. 

The assumption that $\delta M_{k,i}$ forms a martingale difference sequence is a common condition, which was introduced by the original Robbins-Monro method~\cite{robbins1985stochastic}. In particular, the martingale difference sequence assumption holds when the stochastic gradient estimate is obtained from a subset  $\xi\subseteq \{1,...,M\}$ of mini-batches sampled uniformly from the set of size $|\xi|$ of subsets of $\{1,...,M\}$, independently at each iteration. This results in independent noise at each iteration being applied to the stochastic subgradient term. From these mini-batches $\xi$, a subgradient is taken for each $j\in\xi$ and averaged; i.e., 
\begin{equation}\label{eq:subdifest}
g(x,\xi) = \frac{1}{|\xi|}
\sum_{j\in \xi} g^j(x)
\end{equation} 
where $g^j(x)\in \partial f_j(x)$.

\section{Relationship to the Probabilistic Model of Asynchronous Computation}

We now relate the assumptions we have presented so far to this model for inconsistent read given in~\cite{cannelli2016asynchronous}.

%Here we give a few more details describing the relation of the probabilistic model of asynchrony to the underlying hardware
%properties, as modeled in~\cite{cannelli2016asynchronous}.

%In this section, we present $k$ as a \emph{global}
%counter, indicating sequential updates of any block among the variables.

%In iteration $k$, the updated iterate $x^{k+1}_{i^k}$ depends on a random vector
%$\zeta^k\triangleq(i^k,d^k,\xi^k)$. The distribution of $\zeta^k$ depends on the underlying scheduling or message passing protocol. We
%use the following formulation, which applies to a variety of architectures.

To this end, define $\underline{\zeta}^{0:t}\triangleq (\underline{\zeta}^0,\underline{\zeta}^1,...,\underline{\zeta}^t)$ to be a sequence of the blocks and minibatches used, as well as the iterate delays. To measure this process, define the sample space $\Omega$ of all sequences $\{\underline{\zeta}^k\}_k$. To define a $\sigma$-algebra on $\Omega$, for every $k \geq 0$ and every $\underline{\zeta}^{0:k}$, consider the cylinder
\begin{align}
    C^k(\underline{\zeta}^{0:k}) \overset{\triangle}{=} \{\omega \in \Omega: \omega_{0:k} = \underline{\zeta}^{0:k}\},
\end{align}
where $\omega_{0:k}$ is the first $k+1$ elements of $\omega$. Denote $\mathcal{C}^k$ as the set of $C^k(\underline{\zeta}^{0:k})$ for all possible values of $\underline{\zeta}^{0:k}$. Let $\sigma(\mathcal{C}^k)$ be the cylinder $\sigma$-algebra generated by $\mathcal{C}^k$, and define for all $k$,
\begin{align}
    \mathcal{F}_k = \sigma(\mathcal{C}^k),~~~\mathcal{F} \overset{\triangle}{=} \sigma(\cup_{t=0}^{\infty} \mathcal{C}^t).
\end{align}
With a probability measure $\mathbb{P}(C^k(\underline{\zeta}^{0:k}))$ satisfying standard assumptions (detailed in \cite{cannelli2016asynchronous}), induces a probability measure $P$ such that $(\Omega,\mathcal{F},P)$ forms a probability space. This provides a means of formalizing the discrete-time stochastic process $\zeta$ where $\underline{\zeta}$ is a sample path of the process.

The conditional probabilities $\mathbb{P}(\zeta^{k+1} = \omega^{k+1}|\zeta^{0:k} = \omega^{0:k})$ are then defined by
\begin{align}
    p(\zeta^{k+1}|\zeta^{0:k}) = \mathbb{P}(\zeta^{k+1} = \omega^{k+1}|\zeta^{0:k} = \omega^{0:k}) = \frac{\mathbb{P}(C^{k+1}(\zeta^{0:k+1}))}{\mathbb{P}(C^k(\zeta^{0:k}))}
\end{align}
%This process takes values on $(S,\Sigma)$ and is defined on the probability space $(\Omega,\mathcal{F},\mathbb{P})$. 
%The $\sigma$-algebra $\mathcal{F}$ is obtained as follows.
%Let the cylinder $C^k(\underline{\zeta}^{0:t})\triangleq \{ \omega\in \Omega: \omega^{0:k}=\underline{\zeta}^{0:t}\}$ and define $\mathcal{F}^k\triangleq \sigma (C^k)$ and
%$\mathcal{F}\triangleq \sigma(\cup^\infty_{t=0}C^t)$ the cylinder $\sigma$-algebra on $\Omega$.

%Consider the conditional distribution of $\zeta^{k+1}$ given $\zeta^{0:k}$,
%\[
%\mathbb{P}(\zeta^{k+1}|\zeta^{0:k})=\frac{\mathbb{P}(C^{k+1}(\zeta^{0:k+1})}{\mathbb{P}(C^k(\zeta^{0:k}))},
%\]
In \cite{cannelli2016asynchronous}, the following assumptions were introduced for the probabilities of
block selection and the delays.

\begin{assumption}\label{as:probsdels}
The process $\{\zeta^k\}_k$ satisfies,
\begin{enumerate}
\item There exists a $\delta$ such that $d^k_j\le \delta$ for all $j$ and $k$. 
Thus each $d^k_j\in \mathcal{D}\triangleq \{0,...,\delta\}^n$.
\item For all $i$ and $\zeta^{0:k-1}$ such that $p_{\zeta^{0:k-1}}(\zeta^{0:k-1})>0$, it holds that,
\[
\sum_{d\in\mathcal{D}} p((i,d,\xi)|\zeta^{0:k-1})\ge p_{min} 
\]
for some $p_{min}>0$.
\item It holds that,
\[
\mathbb{P}\left(\left\{\zeta\in \Omega:\liminf_{k\to \infty} p(\zeta|\zeta^{0:k-1})>0\right\}\right) = 1
\]
\end{enumerate}
\end{assumption}
%The first condition indicates that there is some maximum possible delay in the vectors, that each element of $x$ used in the
%computation of $x^{k+1}_{i^k}$ is not too old. 

The first condition is an irreducibility condition that there is a positive probability for any block or minibatch to be chosen, given any state of previous realizations of $\{\zeta^k\}$. The second assumption indicates that the set of events in $\Omega$ that asymptotically go to zero in conditional probability are of measure zero.% on the whole space $(\Omega,\mathcal{F},\mathbb{P})$.

%In order to enforce global convergence, we wish to use a diminishing step-size.
%However, at the same time, as synchronization is to be avoided, there must not be a global counter indicating
%the rate of decrease of the step-size. In particular, each core will have its own local step size $\gamma^{\nu(c^k,k)}$
%where $c^k$ is the core, and, defining the random variable $Z^k$ as the component of $\{1,...,\bar{c}\}$ that is active
%at iteration $k$, the random variable denoting the number of updates performed by core $c^k$, denoted by $\nu(k)$ 
%is given by $\nu(k)\triangleq \sum_{j=0}^k I(Z^j=c^k)$.

%In addition, noting that it has been observed that in practice, partitioning variable blocks across cores
%is more efficient than allowing every processor to have the ability to choose across every variable block~\cite{liu2015asynchronous}. Thus
%we partition the blocks of variables across cores. We can thus denote $c^k$ as being defined uniquely by $i^k$, the block variable
%index updated at iteration $k$.

We note that $\liminf_{k\to\infty}\frac{\gamma^{\nu(c^k,k)}}{k} = 0$ in probability is implied by 
\[
\sum_{i\in c^k,d\in\mathcal{D},\xi\subseteq \{1,...,M\}} Pr ((i,d,\xi)|\zeta^{0:k-1})) \to 0
\]
for some subsequence, which is antithetical to Assumption~\ref{as:probsdels}(i). Thus, the step-sizes $\gamma^{\nu(c^k,k)}$ satisfy
\begin{equation}\label{eq:liminfstep}
\liminf_{k\to\infty}\frac{\gamma^{\nu(c^k,k)}}{k} > 0,
\end{equation}
where the limit of the sequence is taken in probability. This is also an assumption for the analysis of stochastic asynchronous parallel algorithms in the simplified model developed in \cite{borkar2008stochastic}.

%To cope with more general unstructured nonsmooth objective functions in this paper, it is necessary to place further assumptions on the continuous-time delays between updates as in Assumptions~\ref{as:uniformintdt} and \ref{as:delaymean}. We note that the integrability condition in Assumption~\ref{as:uniformintdt} is related to the assumptions in \cite{cannelli2016asynchronous}. In particular, Assumption~\ref{as:probsdels} can be viewed as a strengthening of Assumption~\ref{as:uniformintdt}.

\section{Proofs of the Theoretical Results}\label{app:sec:proof}

\subsection{Preliminaries}

\begin{lemma} 
It holds that $\{Y_{k,i},Y^\sigma_{k,i};k,i\}$ is uniformly integrable. Thus, so is 
\[
\left\{\sum_{j=1}^k m^{k-j} Y_{j,i},\sum_{j=1}^k m^{k-j} Y^\sigma_{j,i};k,i\right\}
\]
\end{lemma}
\begin{proof}
%We need to show that there exists an $M$ such that $\mathbb{E}[|Y_{k,i}|]\le M$ and for every $\epsilon$ there
%exists a $\delta>0$ such that, for every $A$ such that $P(A)\le \delta$ and every $Y_{k,i}$, it holds that
%$\mathbb{E}[|Y_{k,i}|:A]\le \epsilon$. 
%~\ref{as:subg}
Uniform integrability of $\{Y_{k,i},Y^\sigma_{k,i};k,i\}$ follows from Assumption 3.2, part 3. 
The uniform integrability of $\left\{\sum_{j=1}^k m^{k-j} Y_{j,i},\sum_{j=1}^k m^{k-j} Y^\sigma_{j,i};k,i\right\}$
follows from $0<m<1$ and the fact that a geometric sum of a uniformly integrable sequence is uniformly
integrable.
\end{proof}

\begin{lemma}\label{lem:gammadbound}
It holds that, for any $K>0$, and all $l$,
\[
\sup_{k<K} \sum_{j=k-[d^k_i]_l}^k \gamma^\sigma_{j,i} \to 0
\]
in probability as $\sigma\to\infty$.
\end{lemma}
\begin{proof}
As $\sigma\to\infty$, by the definition of $\gamma^\sigma_{k,i}$, $ \gamma^\sigma_{k,i}\to 0$ and since by Assumption 3.1 $\max d^{k}_i\le \delta$, 
for all $k<K$, $\sum_{j=k-[d^k_i]_l}^k \gamma^\sigma_{j,i}\le \delta\gamma^\sigma_{k-\delta,i}\to 0$.
\end{proof}
%~\ref{as:probsdels}

Our main convergence result establishes weak convergence of trajectories to an invariant set. Weak convergence is defined in terms of the Skorohod topology, a technical topology weaker than the topology of uniform convergence on bounded intervals, defined in~\cite{dudley1978central}. Convergence of a function $f_n(\cdot)$ to $f(\cdot)$ in the Skorohod topology is equivalent to uniform convergence on each bounded time interval. We denote by $D^j(-\infty,\infty)$ the $j$-fold product space of real-valued functions on the interval $(-\infty,\infty)$ that are right continuous with left-hand limits, with the Skorohod topology forming a complete and separable metric space.

We present a result indicating sufficient conditions for tightness of a set of paths in $D^j(-\infty,\infty)$.
\begin{theorem}\cite[Theorem 7.3.3]{kushner2003stochastic}
Consider a sequence of processes $\{A_k(\cdot)\}$ with paths in $D^j(-\infty,\infty)$ such that
for all $\delta>0$ and each $t$ in a dense set of $(-\infty,\infty)^j$ there is a compact set
$K_{\delta,t}$ such that, 
\[
\inf_n\mathbb{P}\left[A_n(t)|\in K_{\delta,t}\right]\ge 1-\delta,
\]
and for any $T>0$,
\[
\lim_{\delta\to 0} \limsup_n \sup_{|\tau|\le T} \sup_{s\le \delta} \mathbb{E}\left[
\min\left[|A_n(\tau+s)-A_n(\tau)|,1\right]\right]=0
\]
then $\{A_n(\cdot)\}$ is tight in $D^j(-\infty,\infty)$.
\end{theorem}
If a sequence is tight then every weak sense limit process is also a continuous time
process. We say that $A_k(t)$ \emph{converges weakly} to $A$ if,
\[
\mathbb{E}\left[F(A_k(t))\right]\to \mathbb{E}\left[F(A(t))\right]
\]
for any bounded and continuous real-valued function $F(\cdot)$ on $\mathbb{R}^j$. 

Finally, an invariant set for a differential inclusion (DI) is defined as follows.
\begin{definition}
A set $\Lambda\subset \mathbb{R}^n$ is an invariant set for a DI $\dot{x}\in g(x)$ if for all
$x_0\in \Lambda$, there is a solution $x(t)$, $-\infty<t<\infty$ that lies entirely in $\Lambda$
and satisfies $x(0)=x_0$.
\end{definition}

\subsection{Proof of Theorem 3.1}

\begin{remark}
The proof of Theorem 3.1 largely follows an analogous result in Chapter 12 
of~\cite{kushner2003stochastic}, which
considers a particular model of asynchronous stochastic approximation. As we focus on a different model for asynchronous updates, some of the details of the proof are now different. We have also modified the proof to account for momentum. 
\end{remark}
%to the algorithm, and thus in the next section we indicate how to
%treat the distinctions in the proof and show that the result still holds.

By~\cite[Theorem 8.6, Chapter 3]{ethier2009markov}, a sufficient condition for tightness of a sequence $\{A_n(\cdot)\}$ is that for each $\delta>0$ and each $t$ in a dense set in $(-\infty,\infty)$, there is a compact set $K_{\delta,t}$ such that $\inf_n \mathbb{P}[A_n(t)\in K_{\delta,t}]\ge 1-\delta$ and for each positive $T$, 
\[
\lim_{\delta\to 0}\limsup_n \sup_{|\tau|\le T, \,s\le \delta}
\mathbb{E}\left[|A_n(\tau+s)-A_n(\tau)|\right]=0.
\]

Now since $Y_{k,i}$ is uniformly bounded, and $Y^\sigma_{k,i}(\cdot)$ is its interpolation with jumps only at $t$ being equal
to some $T_{k,i}$, it holds that for all $i$,
\[
\lim_{\delta\to 0}\limsup_{\sigma}\mathbb{P}\left[ \sup_{t\le T, \,s\le \delta}
|Y^\sigma_{k,i}(t+s)-Y^\sigma_{k,i}(t)|\ge \eta\right]=0
\]
and so by the definition of the algorithm,
\[
\lim_{\delta\to 0}\limsup_{\sigma}\mathbb{P}\left[ \sup_{t\le T, \,s\le \delta}
|x^\sigma_{k,i}(t+s)-x^\sigma_{k,i}(t)|\ge \eta\right]=0
\]
which implies,
\[
\lim_{\delta\to 0}\limsup_{\sigma}\mathbb{E}\left[ \sup_{t\le T, \,s\le \delta}
|x^\sigma_{k,i}(t+s)-x^\sigma_{k,i}(t)|\right]=0
\]
and the same argument implies tightness for $\{\tau^\sigma_i(\cdot),N^\sigma_{i}(\cdot)\}$ by the uniform boundedness of $\{\delta \tau^\sigma_{i,k}\}$ and 
bounded, decreasing $\gamma^\sigma_{k,i}$ and positive $u^\sigma_{k,i}(x,\psi^\sigma_{k+1,i})$, along with Assumption 2.4.
%~\ref{as:delaymean}. 
Lipschitz continuity follows from the properties of the interpolation functions. Specifically, the Lipschitz constant
of $x^\sigma_{i}(\cdot)$ is $B_g$.

All of these together imply tightness of $\hat x^\sigma_i(\cdot)$ as well. Thus,
\[
\{x^\sigma_i(\cdot),\tau^\sigma_i(\cdot),\hat x^\sigma_i(\cdot),N^\sigma_i(\cdot);\sigma\}
\]
is tight in $D^{4n}[0,\infty)$. 
This implies the Lipschitz continuity of the subsequence limits with probability one, which exist in the
weak sense by Prohorov's Theorem, Theorems 6.1 and 6.2~\cite{billingsley2013convergence}. 

As $\sigma\to\infty$ we denote the limit of the weakly convergent subsequence by,
\[
(x_i(\cdot),\tau_i(\cdot),\hat x_i(\cdot),N_i(\cdot))
\]
Note that,
\[
\begin{array}{l}
x_i(t)=\hat x_i(\tau_i(t)),\\
\hat x_i(t)=x_i(N_i(t)),\\
N_i(\tau_i(t))=t.
\end{array}
\]

For more details, see the proof of ~\cite[Theorem 8.2.1]{kushner2003stochastic}.

%\emph{Due to asymptotic continuity of $x^\sigma_i(\cdot)$ and $\hat x^\sigma_i(\cdot)$ and Lemma~\ref{lem:gammadbound}, 
%the properties of their limits is not affected by the delays.}
Let,
\[
\begin{array}{l}
M^\sigma_i(t) = \sum_{k=0}^{k=p(\sigma)} (1-m)\delta \tau_{k,i}  \left(\sum_{j=0}^k m^j \delta M^\sigma_{k-j,i}\right) \\
\tilde G_i^\sigma(t) = \sum_{k=0}^{k=p(\sigma)}  \delta \tau_{k,i}\left[(1-m)\sum_{j=0}^k m^{j} g_{i} ((x^\sigma_{{k-j}-[d^{k-j}_i]_1,1}(t),...,\right.\\
\qquad\qquad \left. x^\sigma_{{k-j}-[d^{k-j}_i]_j,j}(t),...,x_{{k-j}-[d^{k-j}_i]_N,N}))(t) - 
 g_{i} (\hat x^\sigma_i(t))\right]\\
\bar G_i^\sigma(t) = \sum_{k=0}^{k=p(\sigma)}  \delta \tau_{k,i} g_{i} (\hat x^\sigma(t)) \\
	B_i^{\sigma}(t) = \sum_{k=0}^{\rho(\sigma)} (1 - m)\delta\tau_{k,i} \left(\sum_{j=0}^k m^j\beta_{k-j,i}^{\sigma}\right)\\
W^\sigma_i(t) = \hat x^\sigma_{i}(\tau^\sigma_i(t))-x^\sigma_{i,0}-\bar G_i^\sigma(t) = \tilde G_i^\sigma(t) +M^\sigma_i(t)
\end{array}
\]

Now for any bounded continuous and real-valued function $h(\cdot)$, an arbitrary integer $p$, and $t$ and $\tau$, and $s_j\ge t$ real, we have
\[
\begin{array}{l}
\mathbb{E}\left[h(\tau^\sigma_i(s_j),\hat x^\sigma(\tau^\sigma_i(s_j))\left(W^\sigma_i(t+\tau)-W^\sigma_i(t)\right)\right] \\
\qquad - \mathbb{E}\left[h(\tau^\sigma_i(s_j),\hat x^\sigma(\tau^\sigma_i(s_j))\left(\tilde G^\sigma_i(t+\tau)-\tilde G^\sigma_i(t)\right)\right] \\
\qquad 
- \mathbb{E}\left[h(\tau^\sigma_i(s_j),\hat x^\sigma(\tau^\sigma_i(s_j))\left(M^\sigma_i(t+\tau)-M^\sigma_i(t)\right)\right] \\
\qquad 	-\mathbb{E}[h(\tau_i^{\sigma}(s_j),\hat{x}^{\sigma}(\tau_i^{\sigma}(s_j)))(B_i^{\sigma}(t + \tau) - B_i^{\sigma}(t))] = 0.
\end{array}
\]
Now the term involving $M^\sigma$ equals zero from the Martingale property. The term involving
$B^\sigma$ is zero due to Assumption 2.5.

We now claim that the term involving $\tilde G^\sigma_i$ goes to zero as well. Since $x^\sigma_{k,i}\to x^\sigma_i$ it holds that,
by Lemma~\ref{lem:gammadbound}, $(x^\sigma_{k-[d^k_i]_1,1}(t),...,x^\sigma_{k-[d^k_i]_j,j}(t),...,x^\sigma_{k-[d^k_i]_N,N})\to \hat x^\sigma(t)$
as well. By the upper semicontinuity of the subgradient, it holds that there exists a $g_{i} (\hat x^\sigma_i(t))\in\partial_i f(\hat x^\sigma_i(t))$
such that 
\[
g_{i} ((x^\sigma_{k-[d^k_i]_1,1}(t),...,x^\sigma_{k-[d^k_i]_j,j}(t),...,x_{k-[d^k_i]_N,N}))(t)\to g_{i} (\hat x^\sigma_k(t))
\]
as $\sigma\to\infty$. Thus each term in the sum converges to $g_{i} (\hat x^\sigma_{k-j}(t))$. Now, given
$j$, as $k\to \infty$, the boundedness assumptions and stepsize rules imply that $g_{i} (\hat x^\sigma_{k-j}(t))\to g_{i} (\hat x^\sigma_{k}(t))$.
On the other hand as $k\to \infty$ and $j\to \infty$, $m^{j} g_{i} (\hat x^\sigma_{k-j}(t))\to 0$. Thus 
\[
\sum_{j=0}^k m^j g_{i} (\hat x^\sigma_{k-j}(t))
\to \frac{1-m^k}{1-m} g_{i} (\hat x^\sigma_{k}(t)) \to \frac{1}{1-m} g_{i} (\hat x^\sigma_{k}(t))
\]
and the claim has been shown.

So the weak sense limit of $\lim_{\sigma\to \infty} W^\sigma_i(\cdot)=W_i(\cdot)$ satisfies 
\[
\mathbb{E}\left[h(\tau_i(s_j),\hat x(\tau_i(s_j))\left(W_i(t+\tau)-W_i(t)\right)\right]
\]
and by ~\cite[Theorem 7.4.1]{kushner2003stochastic} is a martingale and is furthermore a constant with probability one by the Lipschitz continuity of $x$ by ~\cite[Theorem 4.1.1]{kushner2003stochastic}. Thus,
\[
W(t) = \hat x(t)-\hat x(0)-\int_0^t g(\hat x(s))ds = 0,
\]
where $g(\hat x(s))\in\partial f (\hat x(s))$, and the conclusion holds. %~\eqref{eq:diprob1} holds. 

\subsection{Convergence with Probability One}\label{app:subsec:convproof}

%We sketch the proof of Theorem~XX as follows. First, we are required to establish that \cite[Assumption 2.1]{dupuis1989stochastic} holds under the assumptions in Theorem~\ref{thrm:main_convergence}.

\begin{theorem}
There exists a functional $\overline{S}(x,T,\phi)$ defined for $x \in \mathbb{R}^n,T > 0,\phi(\cdot) \in C[0,T]$ satisfying:
\begin{enumerate}
    \item[(i)] There exists $\overline{b}(x)$ such that $\overline{S}(x,T,\phi) = 0$ if and only if $\dot{\phi} \in \overline{b}(\phi)$ almost surely, $\phi(0) = x$.
    \item[(ii)] $\overline{S}(x,T,\phi) \geq 0$.
    \item[(iii)] Given compact $F \subset \mathbb{R}^n$, $T > 0$, and $s \in [0,\infty)$, the set $\{\phi:\phi(0) \in F,\overline{S}(\phi(0),T,\phi) \leq s\}$ is compact.
    \item[(iv)] Given compact $F \subset \mathbb{R}^n$, $T > 0$, $h > 0$ and $s \in [0,\infty)$,
    \begin{align}
        \limsup_{\sigma \rightarrow \infty} \gamma_{p_i(\sigma),i} \log P(\hat{x}_i^{\sigma}(\cdot)) \not\in N_h(\Phi(\hat{x}_i^{\sigma}(0),T,s))|\mathcal{F}_{p_i(\sigma),i}) \leq -s
    \end{align}
    uniformly\footnote{For a sequence $\{f_n\}$ and some $f$, $\limsup_{n \rightarrow \infty} f_n \leq f$ in uniformly in a parameter $\alpha$ if for each $\epsilon > 0$, there is $n_{\epsilon} < \infty$ such that $\sup_{n \geq n_{\epsilon}} f_n \leq f+ \epsilon$.} in $\hat{x}_i^{\sigma}(0) \in F$
\end{enumerate}
\end{theorem}

\begin{proof}
Since the noise is bounded by Assumption 2.2, \cite[Theorem 4.1]{dupuis1989stochastic} and \cite[Theorem 5.3]{dupuis1989stochastic} hold without any changes. 
\end{proof}
%~\ref{ass:bounded}
The next step is to modify \cite[Theorem 3.1]{dupuis1989stochastic}. The main change is to replace \cite[Assumption 2.2]{dupuis1989stochastic} with the existence of a solution to the DI in Theorem 3.1 and an invariant set. The impact of this change on the proof is to replace any mention of the trajectory lying in a ball around a limit point with a ball around the invariant set. The result of this modification is then that for each $i$ under Theorem 3.1 as $\sigma \rightarrow \infty$, $\hat{x}_i^{\sigma}(\cdot)$ converges to an invariant set of the DI in Theorem 3.1 with probability one. 

\subsection{Properties of the Limit Point}\label{app:subsec:limpoint}

Finally, we wish to characterize the properties of this invariant set. Training neural networks induces optimization problems with an objective of the form $f(x) = \ell(y_j,a_L)$, where $\ell(\cdot)$ is a standard loss function, and $a_L$---depending on the activation functions from each layer $a_i,~i=1,\ldots,L$---is obtained recursively via $a_i = \rho_i(V_i(x)a_{i-1}),~i = 1,\ldots,L$. These activation functions are typically defined to be piece-wise continuous functions constructed from $\log x$, $e^{x}$, $\max\{0,x\}$ or $\log(1 + e^x)$. In this case, it follows from \cite[Corollary 5.11]{davis2018stochastic} that the set of invariants $\{x^*\}$ satisfy $0 \in \partial f(x^*)$. Moreover, for any iterative algorithm generating $\{x^k\}$ such that $x^k \rightarrow x^*$, the values $f(x^k)$ converge to $f(x^*)$. 

Theorem 3.1 and Theorem 3.2 provide guarantees on convergence to block-wise stationary points; i.e., for each $i$, $0 \in \partial_i f(x)$. As such, every stationary point is also block-wise stationary, which implies that $0\in\partial f(x)$ implies $0\in\partial_i f(x)$ for all $i$. %In practice, the set of block-wise stationary points which are not stationary is not large. 

To ensure convergence to a stationary point, a variant of Algorithm~\ref{alg:assd} can be used, in particular \asm, where every core updates the entire vector (no block partitioning), but locks the shared memory whenever the core either reads or writes. The analysis in this Section also applies to this case. In particular, $i^k=\{1,...,n\}$ for all $k$ 
and every limit of $x^\sigma(t)$ as either $\sigma\to\infty$ or $t\to\infty$ is a critical point of
$f(x)$ and, with probability one, asymptotically the algorithm converges to a critical point of $f(x)$ (i.e., $x$ such that
$0\in\partial f(x)$).

\section{Experiment Description and Additional Numerical Results}\label{app:sec:experiment}
\paragraph{System Detail.} Referring to the settings \textbf{S1} and \textbf{S2} as in Section 4 of the main submission, the following detail apply:
\begin{itemize}[leftmargin=0.1cm,align=left,labelwidth=\parindent,labelsep=4pt]
	\item [\textbf{S1}]: A 
	workstation with a single socket Intel(R) Xeon(R) CPU E5-1650 v4 running @ 3.60 GHz with 6 physical cores amounting to 12 logical cores with hyperthreading. A single Nvidia GeForce RTX 2080 Ti GPU is connected to the motherboard via a PCIe GEN 1@16x interface which allows data transfer speed up to 4 GB/s. We use 4 concurrent processes with 3 threads each to fetch data from the main-memory to the GPU in \sgd and for asynchronous methods each process uses 3 threads and no multiprocessing for data transfer. Clearly, the CPU resources are optimally allocated, which we check before starting the experiments. The system has 64 GB main memory sufficient to buffer the entire dataset for CIFAR10/CIFAR100/SVHN. It runs on Ubuntu 18.04.4, Linux 4.15.0-101-generic operating system. 
	\item [\textbf{S2}]: A NUMA
	workstation with two sockets containing Intel(R) Xeon(R) CPU E5-2650 v4 running @ 2.20 GHz with 12 physical cores amounting to 24 logical cores apiece with hyperthreading. 10 Nvidia GeForce RTX 1080 Ti GPUs are connected to the same motherboard via PCIe GEN 3@16x interfaces which can enable data transfer between the main-memory and GPU up to 16 GB/s. We use 2 concurrent processes with 2 threads each to fetch data from the main-memory to the GPU by decentralized \sgd processes running on each of the GPUs; for asynchronous methods each parameter server process uses 2 threads and no multiprocessing for data transfer, whereas its slaves utilize 2 concurrent processes with 2 threads each for the same. The system has 256 GB main memory. It runs on Debian GNU/Linux 9.12 operating system.   
\end{itemize}

\paragraph{Restricted Backpropagation.} Consider the shared-memory system with $c$ concurrent processes and a model of size $w$ with requirement of $F$ flops for a backpropagation over its computation graph. As described before, we distribute roughly $w/c$ parameters as in a partition to each of the processes. Now, the process assigned with the partition that includes the input layer, has to pay the cost of full backpropagation i.e. $F$. Traversing from the partition that includes the input layer to the one that includes the output layer, the saving by the processes would be $0, \frac{F}{c}, \frac{2F}{c}, \ldots, \frac{(c-1)F}{c}$ flops per backpropagation. Summing it up, we have total savings of $F\frac{(c-1)}{2}$ flops in $c$ concurrent backpropagation steps by $c$ processes. It is imperative that on average \pasm saves roughly $\frac{3F}{8}$ flops per backpropagation carried out by 4 concurrent processes in our experiments.

%Thus, on average, there is a saving of $F\frac{(c-1)(c-2)}{2c}$ flops per backpropagation aggregating over all the processes in the system.

\paragraph{Additional Experiment.} Now, we present an example in which even \pasm itself recovers the classification accuracy achieved by the baseline, see Figure \ref{fig:svhnrn34} and the description thereof.
\begin{figure*}[t]
\centering
\adjustbox{valign=b}{\subfigure[Performance Summary.]{%
\scriptsize\vspace{0pt}
\begin{tabular}{p{0mm}p{17mm}|p{3.mm}p{3.5mm}p{3.5mm}p{3.5mm}p{3.5mm}}
	\toprule\scriptsize
	{} & Algo & Train{\newline}Loss & Train{\newline}Acc@1 & Test{\newline}Loss & Test{\newline}Acc@1 & Time{\newline}(Sec) \\
	\midrule
(1) &          \asm &     0.001 &     100.00 &    0.234 &     93.95 &     1332 \\
(2) &         \pasm &     0.029 &      99.47 &    0.221 &     \textbf{93.75} &      \textbf{726} \\
(3) &           \sgd &     0.001 &     100.00 &    0.267 &     93.30 &      964 \\
(4) &           \sgd (BS:1024) &     0.001 &     100.00 &    0.246 &     93.83 &     1366 \\
\bottomrule
\end{tabular}

}}
\hfill
\adjustbox{valign=b}{\subfigure[Train Loss.]{%
\includegraphics[width=0.25\textwidth]{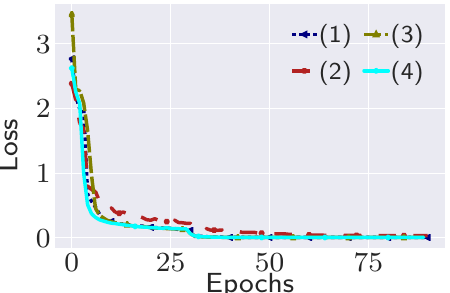}
}}
\hfill
\adjustbox{valign=b}{\subfigure[Top1 Val Accuracy.]{%
	\includegraphics[width=0.25\textwidth]{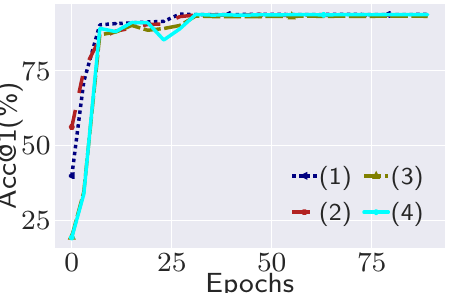}
}}\caption{\small We consider Resnet34 \cite{he2016deep} with 21289802 parameters, training over SVHN \cite{netzer2011reading} dataset. This training experiment has a relatively small memory footprint of $\sim$1.3 GB with BS=256 for \sgd. We train it on the setting \textbf{S1}, see Section 4 in the main submission. The asynchronous methods train with 4 concurrent processes and compute subgradients at BS=256. The baseline minibatch \sgd too computes subgradients at BS=256, whereas the large minibatch method computes subgradients at BS=1024. We give LR warm-up to the large minibatch method. The initial LR, weight-decay and momentum are identical across the methods. Each of the training instances are run for 90 epochs. The LR is dampened by $\frac{1}{10}$ at 30, 60 and 80 epochs and no other variation. In this case, \pasm itself was sufficient to recover the generalization accuracy at par with the baseline and outperforming the large minibatch training. Compared to the large minibatch method, \pasm provides speed-up of 1.33x, whereas compared to the baseline its speed-up is 1.88x. This example shows that in those cases, where we have a relatively deeper model and relatively smaller image sizes -- SVHN dataset contains 73257 training images and 26032 test images of small cropped digits of size 32x32, \pasm itself is sufficient to train the model with classification accuracy at par with the baseline without applying the heuristic approach.}\label{fig:svhnrn34}
\end{figure*}

% Notice that, on average, at each iteration aggregating over 10 GPUs, \pasmp has a slightly smaller BS (1248) compared to \sgd (1280). 

\paragraph{Hyperparameters' Details.} We had already mentioned the minibatch-size and the number of processes or parameter-servers used in the experiments before. The remaining hyperparameters of the experiments are provided in Table \ref{tab:hp}.
\begin{table}[!h]
	\small
	\tabcolsep=0.19cm
	\renewcommand{\arraystretch}{1.5}
\begin{tabular}{|c|c|c|c|c|c|c|}
	\hline
Experiment & Epochs	& Init. LR & $\gamma$  & Dampening at & Weight decay & \pasmp switch at \\
	\hline
\textbf{S1}: ResNet20/CIFAR10/100	& 300  &  0.1  & 0.1 & 150, 225 & 0.0005 & 75, 135, 165, 210, 225 \\
	\hline
\textbf{S1}: Shufflenet/CIFAR100	& 200  &  0.1  & 0.2 & 60, 120, 160 & 0.0001 & 30, 50, 70, 110, 130, 150, 170 \\
\hline
\textbf{S1}: WideResNet16x8/CIFAR10	& 200  &  0.1  & 0.2 & 60, 120, 160 & 0.0001 & 30, 50, 70, 110, 130, 150, 170 \\
\hline
\textbf{S2}: Densenet121/CIFAR100	& 200  &  0.1  & 0.2 & 60, 120, 160 & 0.0001 & 30, 50, 70, 110, 130, 150, 170 \\
\hline
\textbf{S1}: ResNet34/SVHN	& 90  &  0.1  & 0.1 & 30, 60, 80 & 0.0005 & -- \\
	\hline
\end{tabular}\label{tab:hp}\caption{Hyperparameters used in the experiments.}
\end{table}

In each of the experiments, we have a constant momentum of 0.9. \pasmp starts with \asm iterations i.e. computes and applies full subgradient updates, thereafter switches to \pasm and back to \asm at the epoch counts as mentioned in Table \ref{tab:hp}. In large minibatch training, the warm-up procedure starts the learning rate as the initial learning rate, say $\alpha$, of the baseline and gradually increments it to $K\times \alpha$, where $K=\frac{large~minibatch-size}{baseline~minibatch-size}$.

\end{document}